\documentclass[a4paper]{amsart}

\usepackage[english]{babel}
\usepackage{amsthm, latexsym, gastex, amssymb,url,listings}
\usepackage[mathcal]{eucal}
\usepackage{graphicx,url}

\copyrightinfo{2013}{Simone Ugolini}

\DeclareMathOperator{\Irr}{Irr}

\renewcommand{\phi}[0]{\varphi}
\renewcommand{\theta}[0]{\vartheta}
\renewcommand{\epsilon}[0]{\varepsilon}

\newcommand{\Pro}{\text{$\mathbf{P}^1$}}
\newcommand{\F}{\text{$\mathbf{F}$}}

\newtheorem{theorem}{Theorem}[section]
\newtheorem{lemma}[theorem]{Lemma}
\newtheorem{corollary}[theorem]{Corollary}

\theoremstyle{definition}

\newtheorem{definition}[theorem]{Definition}
\newtheorem{example}[theorem]{Example}

\theoremstyle{remark}
\newtheorem{remark}[theorem]{Remark}

\numberwithin{equation}{section}

\begin{document}

\bibliographystyle{amsplain}

\date{}

\title[]
{Sequences of irreducible polynomials without prescribed coefficients over odd prime fields}

\author{S.~Ugolini}
\email{sugolini@gmail.com} 

\begin{abstract}
In this paper we construct infinite sequences of monic irreducible polynomials with coefficients in odd prime fields by means of a transformation introduced by Cohen in 1992. We make no assumptions on the coefficients of the first polynomial $f_0$ of the sequence, which belongs to $\F_p [x]$, for some odd prime $p$, and has positive degree $n$. If $p^{2n}-1 = 2^{e_1} \cdot m$ for some odd integer $m$ and  non-negative integer $e_1$, then, after an initial segment $f_0, \dots, f_s$ with $s \leq e_1$, the degree of the polynomial $f_{i+1}$ is twice the degree of $f_i$ for any $i \geq s$.  
\end{abstract}

\maketitle
\section{Introduction}
Constructing irreducible polynomials of arbitrary large degree over finite fields is desirable for many applications. One way to do that is by means of infinite sequences of irreducible polynomials. More precisely, one considers an irreducible polynomial $f_0 \in \F [x]$, where $\F$ is a finite field, and constructs inductively a sequence of irreducible polynomials $\{ f_i \}_{i \geq 0}$ via a transform that takes the $i$-th polynomial $f_i$ to $f_{i+1}$. Since the goal is getting irreducible polynomials of large degree, such sequences are constructed with the aim that the degree of $f_{i+1}$ is greater than the degree of $f_i$, possibly for any $i \geq 0$. 
Many transformations have been used to produce such sequences of irreducible polynomials. For example, Chu \cite{chu} used the transformation which takes a polynomial $f(x)$ to $f(x^3-3x)$, while more recently Garefalakis \cite{gar} and Stichtenoth and Topuzo\u{g}lu \cite{sti} employed the family of transformations which take a polynomial $f(x)$ of degree $n$ to $(bx+d)^n \cdot f\left( \frac{ax+c}{bx+d} \right)$, being $a,b,c$ and $d$ elements of the field of coefficients of the polynomial $f$.

In \cite{mey} Meyn relied upon the $Q$-transform to produce sequences of irreducible polynomials with coefficients in the field $\F_2$ with two elements. We remind that, if $f \in \F_2 [x]$ and has degree $n$, then its $Q$-transform is the polynomial $f^Q (x) = x^n \cdot f(x+x^{-1})$ of degree $2n$.
In \cite{mey} Meyn proved the following result.

\begin{theorem}
If $f(x) =x^n +a_{n-1} x^{n-1} + \dots + a_1 x + 1$ is an irreducible polynomial of degree $n$ over $\F_2$ such that $a_{n-1} = a_1 = 1$, then $f^Q (x) = x^{2n} + b_{2n-1} x^{2n-1} + \dots + b_1 x + 1$ is a self-reciprocal irreducible polynomial of $\F_2 [x]$ of degree $2n$. 
\end{theorem}

In \cite{seq2}, inspired by the work of Meyn, we considered what happens if the coefficients $a_{n-1}$ and  $a_1$ of the polynomial $f$ are not as in the hypotheses of the just cited theorem. 

Later to Meyn's work, Cohen \cite{coh} dealt with sequences of polynomials over finite fields of odd characteristic. Moving from the consideration that Meyn had limited success using the $Q$-transform in odd characteristic, Cohen introduced another transform, which he called the operator $R$. Such an operator takes any monic polynomial $f$ of degree $n$ over a field of odd characteristic to the monic polynomial
\begin{equation*}
f^R (x) = (2x)^n \cdot f \left( \dfrac{1}{2} (x+x^{-1}) \right)
\end{equation*}  
of degree $2n$. Cohen abbreviated the product $f(1) \cdot f(-1)$ to $\lambda(f)$ and then proved the following result.

\begin{theorem}\label{intro_thm_1}
Let $f_0 (x)$ be a monic irreducible polynomial of degree $n \geq 1$ over $GF(q)$, $q$ odd, where $n$ is even if $q \equiv 3 \pmod{4}$. Suppose also that $\lambda (f_0)$ is a non-square in $GF(q)$. For each $m \geq  1$ define $f_m$ by
\begin{equation*}
f_m (x) = f_{m-1}^R (x).
\end{equation*}
Then, for each $m = 1, 2, 3, \dots$, $f_m$ is an irreducible polynomial over $GF(q)$ of degree $n \cdot 2^m$ and order a divisor of $q^{n \cdot 2^{m-1}}+1$.
\end{theorem}

The proof of  Theorem \ref{intro_thm_1} is based upon the following lemma due to Meyn \cite{mey}.

\begin{lemma}\label{intro_lem_1}
Let $f$ be a monic irreducible polynomial over $GF(q)$, $q$ odd. Then $f^Q$ is irreducible over $GF(q)$ if and only if $f(2) f(-2)$ is a non-square in $G(q)$.
\end{lemma}

Some years later Kyuregyan (\cite{kyu1} \and \cite{kyu2}) described possible quadratic transformations 
\begin{equation}\label{intro_eq_1}
P(x) \to (dx^2+rx+h)^n P \left( \frac{ax^2+bx+c}{dx^2+rx+h} \right)
\end{equation}
into the ring $\F_q[x]$, where $q$ is an odd prime power, allowing to construct irreducible polynomials of higher degree from the given polynomial $P(x)$. In \cite{kyu1}, Section 2, the author focuses on the problem of finding conditions under which the quadratic mapping 
\begin{equation*}
P(x) \to F(x) = (H(a,d))^{-1} (dx^2+rx+h)^n P \left( \frac{ax^2+bx+c}{dx^2+rx+h} \right)
\end{equation*}
produces an irreducible polynomial over $\F_q$,
being
\begin{displaymath}
H(a,d) = 
\begin{cases}
a^n & \text{if $d=0$,}\\
d^n P \left( \frac{a}{d} \right) & \text{if $d \not = 0$}.
\end{cases}
\end{displaymath}

In Section 2, Corollaries 1-5 of the same paper many results are furnished on the irreducibility of $F(x)$, given the irreducibility of $P(x)$. We cite here Corollary 5, which is a generalization of Lemma \ref{intro_lem_1}.

\begin{corollary}
Let $P(x) \not = x$ be an irreducible polynomial of degree $n \geq 1$ over $\F_q$. Then, for any $b \in \F_q$, the self-reciprocal polynomial
\begin{equation*}
F(x) = x^n P \left( \frac{x^2+bx+1}{x} \right)
\end{equation*}
is irreducible over $\F_q$ if and only if the element $P(b+2) P(b-2)$ is a non-square in $\F_q$. 
\end{corollary}

In the subsequent Section of \cite{kyu1} different recurrent methods for constructing irreducible polynomials over $\F_q$ are given. Among others, we would like to cite Theorem  3 in \cite{kyu1}, Section $3$.

\begin{theorem}\label{intro_thm_2}
Let $P(x) \not = x$ be an irreducible polynomial of degree $n \geq 1$ over $\F_q$, where $n$ is even if $q \equiv 3 \pmod{4}$, $r, h, \delta \in \F_q$ and $r \not = 0$, $\delta \not = 0$. Suppose that $P\left(\frac{2 \delta- rh}{r^2} \right) P \left( - \frac{2 \delta + rh}{r^2} \right)$ is a non-square in $\F_q$. Define
\begin{equation*}
F_0 (x) = P(x),
\end{equation*}
\begin{equation*}
F_k (x) = \left(2x + \frac{2h}{r} \right)^{t_{k-1}} F_{k-1} \left( \left( x^2 + \frac{4 \delta^2-(hr)^2}{r^4} \right) \Big/ \left( 2x + \frac{2h}{r} \right) \right), k \geq 1,
\end{equation*}
where $t_k = n 2^k$ denotes the degree of $F_k (x)$. Then $F_k(x)$ is an irreducible polynomial over $\F_q$ of degree $n 2^k$ for every $k \geq 1$.
\end{theorem}

By Theorem \ref{intro_thm_1}  and Theorem \ref{intro_thm_2} one can produce an infinite sequence of monic irreducible polynomials starting from a monic irreducible polynomial $f_0$ with coefficients in $GF(q)$,  provided that some hypotheses on the polynomial $f_0$ are satisfied.  Nevertheless, one can still wonder what happens if any assumption on $f_0$ is removed. More precisely, in this paper we concentrate on the construction of infinite sequences of monic irreducible polynomials over odd prime fields, starting from any monic irreducible polynomial $f_0 \in \F_p [x]$, where $p$ is an odd prime. To do that, we will rely upon the same operator $R$ introduced by Cohen.

Firstly we introduce the following notation.  
\begin{definition}\label{intro_def_nu}
If $m = 2^e \cdot k$, for some odd integer $k$ and non-negative integer $e$, then we denote by $\nu_2 (m)$ the exponent of the greatest power of $2$ which divides $m$, namely $\nu_2 (m) = e$.
\end{definition}


Take now a monic irreducible polynomial $f_0 \in \F_p [x]$ of degree $n$, making no assumptions on its coefficients. For $i \geq 0$, if $f_i^{R}$ is irreducible, then we set $f_{i+1} := f_i^{R}$. In the case that $f_i^{R}$ is not irreducible, then it factors as the product of two monic irreducible polynomials of the same degree (Theorem \ref{seq_4}) and we set $f_{i+1}$ equal to one of these two factors. If $\nu_2 (p^{2n}-1) = {e_1}$, then, after an initial segment $f_0, \dots, f_s$ of the sequence with $s \leq e_1$, the degree of $f_{i+1}$ is twice the degree of $f_{i}$ for $i \geq s$. The procedure just described involves at most $e_0+e_1+1$ polynomial factorizations, even though in most cases not more than $e_1+1$ factorizations are required (see Remark \ref{seq_7}). While factorization of polynomials is a burden one would like to avoid, in our case any factorization of a polynomial of degree $2n$ into two polynomials of the same degree $n$ amounts to solving a linear system of at most $n$ linear equations (see Section \ref{note}).

\section{Preliminaries}
Let $f$ be a monic polynomial of degree $n$ defined over the field $\F_p$ with $p$ elements, for some odd prime $p$. Then,
\begin{equation*}
f^R (x) = (2x)^n \cdot f(\theta_{\frac{1}{2}}(x)),
\end{equation*} 
where $\theta_{\frac{1}{2}}$ is the map which can be defined over the projective line $\Pro (\F_q)  = \F_q \cup \{ \infty \}$ of the finite field $\F_q$ with $q=p^n$ element in such a way:
\begin{displaymath}
\theta_{\frac{1}{2}} (x) = 
\begin{cases}
\infty & \text{if $x= 0$ or $\infty$,}\\
\frac{1}{2} (x+x^{-1}) & \text{otherwise}.
\end{cases}
\end{displaymath}

In \cite{SUk} the dynamics of the map $\theta_{\frac{1}{2}}$ was studied. At first we noticed that $\theta_{\frac{1}{2}}$ is conjugated to the square map. Indeed, if $x$ is any element of $\Pro (\F_q)$, then
\begin{equation}\label{conj}
\theta_{\frac{1}{2}} (x) = \psi \circ s_2 \circ \psi (x), 
\end{equation}
where $s_2$ and $\psi$ are maps defined on $\Pro (\F_q)$ as follows:
\begin{displaymath}
s_2 (x) = 
\begin{cases}
x^2 & \text{if $x \in \F_q$,} \\
\infty & \text{if $x = \infty$,}
\end{cases}
\quad
\psi(x) =
\begin{cases}
\dfrac{x+1}{x-1} & \text{if $x \in \Pro (\F_q) \backslash \{1, \infty \}$,} \\
1 & \text{if $x = \infty$,}\\
\infty & \text{if $x=1$.}
\end{cases}
\end{displaymath}

We say that an element $x \in \Pro (\F_q)$ is $\theta_{\frac{1}{2}}$-periodic if and only if $\theta_{\frac{1}{2}}^k (x) = x$ for some positive integer $k$. The smallest such $k$ will be called the period of $x$ with respect to the map $\theta_{\frac{1}{2}}$. Nonetheless, if an element $x \in \Pro (\F_q)$ is not $\theta_{\frac{1}{2}}$-periodic, then it is preperiodic, namely $\theta_{\frac{1}{2}}^l (x)$ is $\theta_{\frac{1}{2}}$-periodic for some positive integer $l$. 

As explained in \cite{SUk}, we can associate a graph $G^q_{\theta_{\frac{1}{2}}}$ with the map $\theta_{\frac{1}{2}}$ over $\Pro (\F_q)$. To do that, we label the vertices of the graph by the elements of $\Pro (\F_q)$. Then, if $\alpha, \beta \in \Pro(\F_q)$ and $\beta = \theta_{\frac{1}{2}} (\alpha)$, we connect with an arrow $\alpha$ to $\beta$.

We notice in passing that, being $\theta_{\frac{1}{2}}$ conjugated with the square map, the graph $G^q_{\theta_{\frac{1}{2}}}$ is isomorphic to the graph associated with the square map in $\Pro(\F_q)$. Indeed, the dynamics of the square map (and more in general of power maps $x^k$) and the structure of the associated graphs over finite fields have been extensively studied (see for example \cite{cho},  \cite{rog}, \cite{sha}, \cite{vas} ). 

The  reader can find more details about the length and the number of the cycles of $G_{\theta_{\frac{1}{2}}}^q$ in \cite{SUk}. In the present paper we are just interested in the structure of the reversed binary tree attached to any vertex of a cycle. For that reason we state here a reduced version of Theorem 2.6 \cite{SUk} just adopting the notation given in Definition \ref{intro_def_nu}.

\begin{theorem}\label{thm_tree}
Let $\alpha \in \Pro(\F_q)$ be a $\theta_{\frac{1}{2}}$-periodic element. If $\alpha \in \{-1, 1 \}$, then $\alpha$ is not root of a tree in $G_{\theta_{\frac{1}{2}}}^q$. If $\alpha \not \in \{-1, 1 \}$, then $\alpha$ is the root of a reversed binary tree of depth $\nu_2 (q-1)$ in $G_{\theta_{\frac{1}{2}}}^q$. Moreover, the root has one child and all the other vertices, except for the leaves, have two children.
\end{theorem}

We fix once for the remaining part of the section an odd prime $p$ and introduce some notation.

\begin{definition}
If $f \in \F_p [x] \backslash \{ x \}$ is a monic irreducible polynomial and $\alpha$ is a non-zero root of $f$ in an appropriate extension of $\F_p$, then we denote by $\tilde{f}$ the minimal polynomial of $\theta_{\frac{1}{2}} (\alpha)$ over $\F_p$.
\end{definition}

\begin{definition}
We denote by $\Irr_p$ the set of all monic irreducible polynomials of $\F_p [x]$ different from $x+1$ and $x-1$. If $n$ is a positive integer, then $\Irr_p(n)$ denotes the set of all polynomials of $\Irr_p$ of degree $n$.
\end{definition}

Consider the following lemma.

\begin{lemma}\label{seq_1}
Let $n$ be a positive integer and suppose that $\nu_2 (p^n-1) \geq 2$. Then, $\nu_2 (p^{2n}-1) = \nu_2 (p^n-1)+1$.
\end{lemma}
\begin{proof}
Since $\nu_2 (p^n-1) \geq 2$, we have that $p^n \equiv 1 \pmod{4}$ and $p^n+1 \equiv 2 \pmod{4}$, namely $\nu_2 (p^n+1)=1$. 
Summing all up,
\begin{displaymath}
\nu_2 (p^{2n}-1) = \nu_2( (p^n-1) \cdot (p^n+1)) = \nu_2(p^n-1)+1.
\end{displaymath}
\end{proof}

\begin{remark}
Notice that we cannot remove the hypothesis $\nu_2 (p^n-1) \geq 2$ in Lemma \ref{seq_1}. In fact, if $\nu_2 (p^n-1)=1$, then anything can happen. Consider for example the primes $23$ and $31$ with $n=1$. We have that $\nu_2 (22) = 1$ and $\nu_2 (30) = 1$. Nevertheless, $\nu_2 (23^2-1)=  4$, while $\nu_2 (31^2-1) = 6$.	 
\end{remark}

We will make use of the following technical lemma in the forthcoming theorem.
\begin{lemma}\label{seq_2}
Let $f$ be a polynomial of positive degree $n$ of $\F_p [x]$. Suppose that $\beta$ is a root of $f$ and that $\beta = \theta_{\frac{1}{2}} (\alpha)$ for some $\alpha, \beta$ in suitable extensions of $\F_p$. Then, $\alpha$ and $\alpha^{-1}$ are roots of $f^{R}$.
\end{lemma}
\begin{proof}
The thesis follows easily evaluating $f^{R}$ at $\alpha$ and $\alpha^{-1}$.  In fact,
\begin{eqnarray*}
f^{R} (\alpha) & = & (2  \alpha)^n \cdot f(\theta_{\frac{1}{2}}(\alpha)),\\
f^{R} (\alpha^{-1}) & = & (2  \alpha^{-1})^n \cdot f(\theta_{\frac{1}{2}}(\alpha))
\end{eqnarray*}
and $f(\theta_{\frac{1}{2}}(\alpha)) = f(\beta) = 0$.
\end{proof}

\begin{theorem}\label{seq_3}
Let $f$ be a polynomial of $\Irr_p(n) \backslash \{ x  \}$ for some positive integer $n$. The following hold.
\begin{itemize}
\item If the set of roots of $f$ is not inverse-closed, then $\tilde{f} \in \Irr_p(n)$.
\item If the set of roots of $f$ is inverse-closed, then $n$ is even and $\tilde{f} \in \Irr_p (n/2)$.
\end{itemize}
\end{theorem}
\begin{proof}
Suppose that the set of roots of $f$ is not inverse-closed. If $\beta = \theta_{\frac{1}{2}}(\alpha)$ for some root $\alpha$ of $f$, then $\beta$ is root of $\tilde{f}$. Since $\alpha$ is root of the polynomial $x^2 - 2 \beta x + 1$ and the degree of $\alpha$ over $\F_p$ is $n$, the degree of $\beta$ over $\F_p$ is either $n$ or $n/2$. Suppose that $\tilde{f}$ has degree $n/2$ (and $n$ is even) and consider the polynomial $g= \tilde{f}^{R}$. The polynomial $g$ has degree $n$ and $\alpha, \alpha^{-1}$ are among its roots. We deduce that $g$ is the minimal polynomial of $\alpha$, namely $g = f$. This implies that the set of roots of $f$ is inverse-closed in contradiction with the initial assumption. Therefore, $\tilde{f} \in \Irr_p(n)$.

Suppose now that the set of roots of $f$ is inverse-closed and consider any root $\alpha$ of $f$. Since $\alpha = \alpha^{-1}$ if and only if $\alpha^2 = 1$, namely $\alpha = \pm 1$, and  $f$ is different from $x+1$ and $x-1$, we conclude that $\alpha \not = \alpha^{-1}$ and the degree of $f$ is even. Let $\beta = \theta_{\frac{1}{2}} (\alpha)$. By definition $\tilde{f}$ is the minimal polynomial of $\beta$. We notice in passing that any root of $f$ is of the form $\alpha^{p^i}$, while any root of $\tilde{f}$ is of the form $\beta^{p^i}$ for some non-negative integer $i$. Moreover, 
\begin{displaymath}
\beta^{p^i} = \theta_{\frac{1}{2}}(\alpha)^{p^i} = \dfrac{1}{2^{p^i}} \cdot \left( \alpha+ \alpha^{-1} \right)^{p^i} = \dfrac{1}{2} \cdot \left(\alpha^{p^i} + \alpha^{-p^i} \right) = \theta_{\frac{1}{2}} (\alpha^{p^i}).
\end{displaymath}
Therefore, the map $\theta_{\frac{1}{2}}$ defines a surjective $2$-$1$ correspondence between the set of roots of $f$ and the set of roots of $\tilde{f}$, implying that the degree of $\tilde{f}$ is $n/2$.
\end{proof}

\begin{theorem}\label{seq_4}
Let $f(x) = x^n+a_{n-1} x^{n-1} + \dots + a_1 x + a_0 \in \Irr_p (n)$ for some positive integer $n$. The following hold.
\begin{itemize}
\item $0$ is not root of $f^{R}$.
\item The set of roots of $f^{R}$ is closed under inversion.
\item Either $f^{R} \in \Irr_p (2n)$ or $f^{R}$ splits into the product of two polynomials $m_{\alpha}, m_{\alpha^{-1}}$ in $\Irr_p (n)$, which are respectively the minimal polynomial of $\alpha$ and $\alpha^{-1}$, for some $\alpha \in \F_{p^n}$. Moreover, in the latter case at least one among $\alpha$ and $\alpha^{-1}$ is not $\theta_{\frac{1}{2}}$-periodic.
\end{itemize}
\end{theorem}
\begin{proof}
Since $f$ is irreducible,  $a_0 = 0$ only if $f(x) = x$. If this is the case, then $f^{R} (x) = x^2+1$ and $f^{R} (0) \not = 0$. If $a_0 \not = 0$, then the constant term of $f^{R}$ is equal to $1$. In fact, since
\begin{displaymath}
f^{R} (x) = (2 x)^n \cdot \left( 2^{-n} (x+x^{-1})^n + a_{n-1} \cdot 2^{-n+1} (x+x^{-1})^{n-1} + \dots + a_0 \right),
\end{displaymath}
we have that the constant term of $f^{R}$ is determined by the expansion of $(x+x^{-1})^n$ only. Therefore, the constant term of $f^{R}$ is equal to $1$ and we deduce that $0$ cannot be a root of $f^{R}$.

If $\alpha$ is a root of $f^{R}$, then $\alpha$ is invertible. Being $f^{R} (\alpha) = 0$ and $\alpha \not = 0$, we get that $f(\theta_{\frac{1}{2}} (\alpha) ) = 0$. Since $f^{R} (\alpha^{-1}) = (2 \alpha^{-1})^n \cdot f(\theta_{\frac{1}{2}}(\alpha))$, we get that also $f^{R} (\alpha^{-1}) = 0$. We conclude that the set of roots of $f^{R}$ is  inverse-closed.

Consider now a root $\alpha$ of $f^{R}$. Since $\alpha \not = 0$, we have that $f ( \theta_{\frac{1}{2}} (\alpha) ) = 0$, namely $\gamma = \theta_{\frac{1}{2}} (\alpha) $ is a root of $f$. Moreover, $\alpha$ is root of the polynomial $x^2 - 2 \gamma x +1$. Summing all up, we conclude that the degree of $\alpha$ over $\F_p$ is either $n$ or $2n$. In the latter case the minimal polynomial of $\alpha$ has degree $2n$ and must be equal to $f^{R}$. In the former case the minimal polynomial $m_{\alpha}$ of $\alpha$ has degree $n$ and by Theorem \ref{seq_3} the set of roots of $m_{\alpha}$ is not inverse-closed (on the contrary, $f$, which is the minimal polynomial of $\theta_{\frac{1}{2}} (\alpha)$, should have degree $n/2$). Hence, using the notation of the claim, $f^{R} (x) = m_{\alpha} (x) \cdot m_{\alpha^{-1}} (x)$. Moreover, we observe that $\theta_{\frac{1}{2}}(x) = \gamma$ if and only if $x = \alpha$ or $\alpha^{-1}$. If $\gamma$ is not $\theta_{\frac{1}{2}}$-periodic, then both $\alpha$ and $\alpha^{-1}$ cannot be $\theta_{\frac{1}{2}}$-periodic too. On the converse, one among $\alpha$ and $\alpha^{-1}$ is $\theta_{\frac{1}{2}}$-periodic, while the other element belongs to the level $1$ of the tree rooted in $\gamma$.
\end{proof}

\section{Constructing sequences of irreducible polynomials}\label{construction}
We fix once for all current section an odd prime $p$ and a positive integer $n$. The following theorem furnishes a procedure for constructing an infinite sequence of irreducible polynomials, starting from any polynomial of $\Irr_p (n)$.
\begin{theorem}\label{seq_6}
Let $f_0 \in \Irr_p (n)$, with $\nu_2 (p^n-1) = {e_0}$ and $\nu_2 (p^{2n}-1) = {e_1}$ for positive integers $e_0, e_1$ with $e_0 < e_1$.

If $f_0^{R}$ is irreducible, define $f_1 := f_0^{R}$. Otherwise, set $f_1$ equal to one of the two monic irreducible factors of degree $n$ of $f_0^R$ having a root which is not $\theta_{\frac{1}{2}}$-periodic, as stated in Theorem \ref{seq_4}.  

For $i \geq 2$ define inductively a sequence of polynomials $\{ f_i \}_{i \geq 2}$ in such a way: if $f_{i-1}^{R}$ is irreducible, then $f_i := f_{i-1}^{R}$; otherwise, set $f_i$ equal to one of the two irreducible factors of degree $n$ of $f_{i-1}^{R}$ as stated in Theorem \ref{seq_4}.

Then, there exist two positive integers $s_1, s_2$ such that:
\begin{itemize}
\item $f_0, \dots, f_{s_1-1} \in \Irr_p(n)$;
\item $f_{s_1}, \dots, f_{s_1+s_2-1} \in \Irr_p(2n)$;
\item $f_{s_1+s_2+j} \in \Irr_p (2^{2+j} n)$ for any $j \geq 0$;
\item $s_1 \leq e_0+1$ and $s_2 = e_1 - e_0$.
\end{itemize}
\end{theorem}
\begin{proof}
Let $\beta_0 \in \F_{p^n}$ be a root of $f_0$. In $G_{\theta_{\frac{1}{2}}}^{p^n}$ the vertex $\beta_0$ lies on the level $k \geq 0$ of some binary tree of depth $e_0$ rooted in an element $\gamma \in \F_{p^n}$. In particular, if $k = 0$, then $\beta_0  = \gamma$. 

If $f_0^{R}$ is irreducible of degree $2n$, then we set $f_1 := f_0^{R}$. The equation $\theta_{\frac{1}{2}}(x) = \beta_0$ has exactly two solutions $\beta_1$ and $\beta_1^{-1}$ in $\F_{p^{2n}}$. By Lemma \ref{seq_2}, $\beta_1$ and $\beta_1^{-1}$ are roots of $f_1 \in \Irr_p (2n)$.  All considered we can say that $\beta_0$ is a leaf of $G_{\theta_{\frac{1}{2}}}^{p^n}$. Therefore, in this case $k=e_0$ and $s_1 = 1$.

If $f_0^{R}$ is not irreducible, then we define $f_1$ as one of the two monic irreducible factors of degree $n$ of $f_{0}^{R}$ having a root $\beta_1$, which is not $\theta_{\frac{1}{2}}$-periodic.
 We prove that, for any integer $i$ such that $0 \leq i \leq e_0-k$, there exists an element $\beta_i \in \F_{p^n}$ such that $\theta_{\frac{1}{2}}^{k+i} (\beta_i) = \gamma$ and $\beta_i$ is a root of $f_i$. Indeed, this is trivially true if $i=0$. Suppose that, for some $i < e_0-k$, there exists an element $\beta_i$ such that $\theta_{\frac{1}{2}}^{k+i} (\beta_i) = \gamma$ and $\beta_i \in \F_{p^n}$ is a root of $f_i$. Since $k+i < e_0$ and the tree rooted in $\gamma$ has depth $e_0$, there exists an element $\beta' \in \F_{p^n}$ such that $\theta_{\frac{1}{2}} (\beta') = \beta_i$. By Lemma \ref{seq_2} the element $\beta'$ is a root of $f_i^{R}$. Since $\beta' \in \F_{p^n}$, the polynomial $f_i^{R}$ splits into the product of two  polynomials $g_1, g_2 \in \Irr_p (n)$. One among $g_1$ and $g_2$ is equal to $f_{i+1}$. Moreover, either $\beta'$ or $(\beta')^{-1}$ is root of $f_{i+1}$. We can say, without loss of generality, that $\beta'$ is root of $f_{i+1}$. Therefore, setting $\beta_{i+1} := \beta'$, we get that $\beta_{i+1}$ is a root of $f_{i+1}$ and $\theta_{\frac{1}{2}}^{k+i+1} (\beta_{i+1}) = \gamma$. Now, consider the polynomial $f_{e_0-k}$. By construction $\beta_{e_0-k} \in \F_{p^n}$ is a root of $f_{e_0-k}$. Moreover, $\beta_{e_0-k}$ is a leaf of the tree of $G_{\theta_{\frac{1}{2}}}^{p^n}$ rooted in $\gamma$. Consider now an element $\beta$ such that $\theta_{\frac{1}{2}} (\beta) = \beta_{e_0-k}$. Since $\beta$ cannot belong to the same tree of $G_{\theta_{\frac{1}{2}}}^{p^n}$, we have that $\beta \in \F_{p^{2n}} \backslash \F_{p^n}$. Therefore, $f_{s_1} := f_{e_0-k+1} = f_{e_0-k}^{R}$ is irreducible of degree $2n$. Since $k \geq 0$, the index $s_1 = e_0-k+1 \leq e_0+1$. 

Now we prove by induction on $i$ that, for any integer $i$ such that $e_0-k+1 \leq i \leq e_1 - k$, there exists an element $\beta_i \in \F_{p^{2n}}$ such that $\theta_{\frac{1}{2}}^{k+i} (\beta_i) = \gamma$ and $\beta_i$ is a root of $f_i$. In virtue of what we have just proved, this is true if $i = e_0-k+1$. Take now an integer $e_0-k+1 \leq i < e_1-k$. By inductive hypothesis there exists an element $\beta_i \in \F_{p^{2n}}$ such that $\theta_{\frac{1}{2}}^{k+i} (\beta_i) = \gamma$ and $\beta_i$ is a root of $f_i$. Since $\beta_i$ belongs to the level $i +k < e_1$ of the tree of $G_{\theta_{\frac{1}{2}}}^{p^{2n}}$ rooted in $\gamma$,  there exists an element $\beta' \in \F_{p^{2n}}$ such that $\theta_{\frac{1}{2}} (\beta') = \beta_i$. Then, one among $\beta'$ and $(\beta')^{-1}$, say $\beta'$, is root of $f_{i+1}$. We set $\beta_{i+1} := \beta'$ and complete the inductive proof. Since $\beta_{e_1-k}$ is a leaf of the tree of $G_{\theta_{\frac{1}{2}}}^{p^{2n}}$ rooted in $\gamma$, any element $\beta$ such that $\theta_{\frac{1}{2}} (\beta) = \beta_{e_1-k}$ cannot belong to $\F_{p^{2n}}$. Then, $f_{s_1+s_2}:=f_{e_1-k+1}$ has degree $4n$. Since $s_1 = e_0-k+1$ we conclude that $s_2 = e_1 - e_0$.

Finally we prove by induction on $j$ that, for any integer $j \geq 0$, there exists an element $\gamma_{s_1+s_2+j} \in \F_{p^{2^{2+j} n}}$ such that:
\begin{itemize}
\item $\gamma_{s_1+s_2+j}$ has degree $2^{2+j} n$ over $\F_p$; 
\item $\gamma_{s_1+s_2+j}$ is a leaf of the tree of $G_{\theta_{\frac{1}{2}}}^{p^{2^{2+j}n}}$ rooted in $\gamma$;
\item $f_{s_1+s_2+j}(\gamma_{s_1+s_2+j})=0$ and consequently $f_{s_1+s_2+j} \in \Irr_p(2^{2+j} n)$.
\end{itemize}

Consider firstly the base step. We denote by $\gamma_{s_1+s_2}$ one of the roots of $f_{s_1+s_2}$ in $\F_{p^{4n}}$. Then, $\gamma_{s_1+s_2}$ has degree $2^2n$ over $\F_p$ and $f(\gamma_{s_1+s_2}) = 0$. We know that $\theta_{\frac{1}{2}} (\gamma_{s_1+s_2}) = \beta_{e_1-k}$ and $\beta_{e_1-k}$ belongs to the level $\nu_2 (p^{2n}-1)$ of the tree of $G_{\theta_{\frac{1}{2}}}^{p^{2n}}$ rooted in $\gamma$. Therefore, $\gamma_{s_1+s_2}$ belongs to the level $e_1+1$ of the tree of $G_{\theta_{\frac{1}{2}}}^{p^{4n}}$ rooted in $\gamma$. In addition, $\gamma_{s_1+s_2}$ is a leaf of the same tree, since the depth of such tree in $G_{\theta_{\frac{1}{2}}}^{p^{4n}}$ is
\begin{equation*}
\nu_2 ( p^{4n}-1 ) = \nu_2 (p^{2n}-1) + 1, 
\end{equation*} 
in accordance with Theorem \ref{thm_tree} and Lemma \ref{seq_1}.

As regards the inductive step, let $j$ be a non-negative integer and consider an element $\gamma_{s_1+s_2+j+1} \in \F_{p^{2^{2+j+1}n}}$ such that $\theta_{\frac{1}{2}} (\gamma_{s_1+s_2+j+1}) = \gamma_{s_1+s_2+j}$. Since $\gamma_{s_1+s_2+j}$ is a leaf of the tree of  $G_{\theta_{\frac{1}{2}}}^{p^{2^{2+j}n}}$ rooted in $\gamma$ by inductive hypothesis, $\gamma_{s_1+s_2+j+1}$ has degree $2^{2+j+1} n$ over $\F_p$. Moreover, being $f_{s_1+s_2+j}^R (\gamma_{s_1+s_2+j+1}) = 0$, it follows that $f_{s_1+s_2+j+1} = f_{s_1+s_2+j}^R$, namely $f_{s_1+s_2+j+1} \in \Irr_p(2^{2+j+1} n)$. To end with, $\gamma_{s_1+s_2+j+1}$ is a leaf of the tree of $G_{\theta_{\frac{1}{2}}}^{p^{2^{2+j+1}n}}$ rooted in $\gamma$, since such a tree has depth
\begin{equation*}
\nu_2 ( p^{2^{2+j+1}n}-1 ) = \nu_2 (p^{2^{2+j} n}-1) + 1, 
\end{equation*} 
being $\nu_2 (p^{2^{2+j} n}-1)$ the depth of the tree of $G_{\theta_{\frac{1}{2}}}^{p^{2^{2+j}n}}$ rooted in $\gamma$.
\end{proof}

\begin{remark}\label{seq_7}
Using the notation of  Theorem \ref{seq_6}, if $f_0^{R}$ is not irreducible, then it splits into the product of two irreducible polynomials $g_1, g_2$ of equal degree. By Theorem \ref{seq_4} one of them must have a root which is not $\theta_{\frac{1}{2}}$-periodic. In principle we do not know which of the two polynomials has this property. Therefore, we just set $f_1$ equal to either $g_1$ or $g_2$. Then we proceed constructing the sequence as stated in the theorem's claim. If none of the polynomials $f_i$, for $i \leq e_0+1$, has degree $2n$, then we break the procedure and set $f_1 := g_2$, which will have a non-$\theta_{\frac{1}{2}}$-periodic root, as stated by Theorem \ref{seq_4} (see also Example \ref{exm_1}).

Summing all up, we must test the irreducibility of a polynomial of degree at most $4n$ over $\F_p$, and eventually factor it, not more than $e_0+e_1+1$ times. Indeed, if $f_0$ has no $\theta_{\frac{1}{2}}$-periodic root or all its roots are $\theta_{\frac{1}{2}}$-periodic  but we pick up the factor of $f_0^R$ which has no $\theta_{\frac{1}{2}}$-periodic root, then the number of factorization is at most $e_1+1$.  
\end{remark}

We conclude this section with two examples of sequences with initial polynomial $f_0$ belonging  to $\Irr_7(1)$.
\begin{example}
Let $f_0 (x) := x \in \Irr_7(1)$. In accordance with the notation of Theorem \ref{seq_6} we have that $e_0 = \nu_2 (6) = 1$ and $e_1 = \nu_2 (48) = 4$.
The polynomial $f_1 (x) := f_0^{R} (x) = x^2+1$ is irreducible.  Since $e_1=4$ and $e_0=1$  we expect that $f_2$ and $f_3$ belong to $\Irr_7 (2)$, while $f_4$ belongs to $\Irr_7 (4)$. 

The polynomial $f_1^{R} (x) = x^4-x^2+1$ splits into the product of two irreducible factors, namely $f_1^{R} (x) = (x^2+2) \cdot (x^2+4)$.  We set $f_2 (x):= x^2+2$. Now, $f_2^{R} (x) = x^4+3 x^2+1$ splits into the product of two irreducible factors as $f_2^{R} (x)=(x^2+3x-1) \cdot (x^2+4x-1)$. We set $f_3 (x) := x^2+3x-1$.

The polynomial $f_3^{R} (x) = x^4-x^3-2x^2-x+1$ is irreducible. Hence we set $f_4 := f_3^{R}$. Now, for $i \geq 3$, any polynomial $f_{i+1} := f_i^{R}$ is irreducible, namely we can construct an infinite sequence of irreducible polynomials whose degree doubles at each step. 
\end{example}

\begin{example}\label{exm_1}
Let $f_0 (x) := x-3 \in \Irr_7 (1)$. Since $\nu_2(6)=1$, using the notation of Theorem \ref{seq_6} we have that $s_1 \leq 2$. Therefore, in the sequence we are going to construct, at most the polynomials $f_0$ and $f_1$ have degree $1$. The polynomial $f_0^{R} (x) = x^2+x+1$ is not irreducible. Indeed, $f_0^{R} (x) = (x-4) (x-2)$. We set $f_1$ equal to one among the two factors of degree $1$ of $f_0^{R}$. For example, set $f_1 (x) := x-4$. The polynomial $f_1^{R} (x) = x^2-x+1$ factors as $f_1^{R} (x) = (x-3)(x-5)$. If we set $f_2$ equal to any of the factors of degree $1$ of $f_1^{R}$ we get that $f_2$ is a polynomial of degree $1$ too. Hence, we break the procedure and change the polynomial $f_1$ as suggested in Remark \ref{seq_7}.

Set $f_1 (x) := x-2$. Now, $f_1^{R} (x) = x^2+3x+1$ is irreducible in $\F_7 [x]$. Therefore we set $f_2 (x) := x^2+3x+1$. Now, $f_2^{R} (x) = x^4-x^3-x^2-x+1$. We notice that $f_2^{R} (x) = (x^2+x+3) \cdot (x^2-2x-2)$, where both the factors of degree two belong to $\Irr_7(2)$. Set $f_3 (x) := x^2+x+3$. Now, $f_3^{R} (x) = x^4+2x^3+2x+1$ splits into the product of two irreducible polynomials of $\Irr_7(2)$, namely $f_3^{R} (x) = (x^2-3x-2) \cdot (x^2-2x+3)$. Set $f_4 (x) := x^2-3x-2$. We have that $f_4^{R} (x) = x^4+x^3+x^2+1 \in \Irr_7(4)$. Therefore we set $f_5 := f_4^{R}$ and, in virtue of Theorem \ref{seq_6}, we are guaranteed that any polynomial $f_{i+1} := f_{i}^{R}$ will be irreducible for $i \geq 4$.
\end{example} 

\section{A note about Theorem \ref{seq_6}} \label{note}
In a generic step of the iterative procedure described in Theorem \ref{seq_6} we have to decide whether the polynomial $f_i^{R}$ is irreducible or not and, in the latter case, factoring it. Dropping the indices off, the problem we are dealing with consists in deciding if, taken an irreducible polynomial $f$ of degree $n$ of $\F_p[x]$ where $p$ is an odd prime and $n$ a positive integer,  the polynomial $f^{R}$ is irreducible or not. In the latter case we have to find two irreducible monic polynomials, say $g_1$ and $g_2$, of degree $n$ such that $f^{R} (x) = g_1(x) \cdot g_2 (x)$. 

If $\beta \in \F_{p^n}$ is a root of $f$, then any element of $\F_{p^n}$ is expressible as 
\begin{equation}\label{note_1}
c_{n-1} \beta^{n-1} + \dots + c_1 \beta + c_0,
\end{equation} 
where $c_{n-1}, \dots, c_1, c_0 \in \F_p$.

If $\alpha$ is a solution of the equation $\theta_{\frac{1}{2}}(x) = \beta$, then $\alpha$ is a root of $f^{R}$. The fact that $\theta_{\frac{1}{2}}(\alpha) = \beta$ is equivalent to saying that $\alpha$ is root of $x^2 - 2 \beta  x +1$, namely 
\begin{equation*}
\alpha = \beta + \sqrt{\beta^2-1}
\end{equation*}
for some square root $\sqrt{\beta^2-1}$ of $\beta^2-1$.
Therefore, either $\alpha \in \F_{p^n}$ or $\alpha \in \F_{p^{2n}}$. In particular,  $\alpha \in \F_{p^n}$ if and only if $\beta^2-1$ is a square in $\F_{p^n}$, namely if and only if $(\beta^2-1)^{\frac{p^n-1}{2}} = 1$ in $\F_{p^n} \cong \F_p [x] / (f)$. If this latter test fails, then we can conclude that $f^{R}$ is irreducible. On the contrary, we can find a square root of $\beta^2-1$ as explained for example in the proof of Lemma 7.7 of \cite{vdw}, which relies upon Theorem VI.6.1 of \cite{lan}. To do that, set $a = \beta^2-1$. Following the steps of the proof we define $A = a^{(p-1)/2}$ and look for a non-zero element $c \in \F_{p^n}$ such that 
\begin{equation}
c^p = A c.
\end{equation}
Since any $c \in \F_{p^n}$ can be expressed as in (\ref{note_1}), solving the last equation amounts to finding the coefficients $c_i \in \F_p$ which satisfy the equation
\begin{equation}\label{note_2}
c_{n-1} \beta^{p \cdot (n-1)} + \dots + c_1 \beta^p + c_0 = A \cdot (c_{n-1} \beta^{n-1} + \dots + c_1 \beta + c_0).
\end{equation}
Any exponent in the powers of $\beta$ can be reduced to a positive integer smaller than $n$, since $f(\beta) = 0$ and $f§$ has degree $n$. Therefore, solving (\ref{note_2}) amounts to finding a solution 
\begin{displaymath}
(c_0, \dots, c_{n-1}) \in \F_p^n
\end{displaymath}
of a linear system of at most $n$ linear equations.
Once we have found such a $c$, we notice that $c^2/a$ is a quadratic residue in $\F_p$. Finally, we find a square root $d$ of $c^2/a$ in $\F_p$ and notice that $c/d$ is a square root of $a$.

Summing all up, $c/d \in \F_{p^n}$ can be expressed as a linear combination of $1$, $\beta$, $\beta^2, \dots, \beta^{{n-1}}$ with coefficients in $\F_p$. Substituting $c/d$ in place of $\sqrt{\beta^2-1}$ we express $\alpha$ as linear combination of the powers $\beta^{i}$ with $0 \leq i \leq n-1$.

To end with, we can factor $f^{R}$ as the product of two irreducible factors $g_1 (x)$, $g_2 (x)$ in $\F_p [x]$ of degree $n$, namely
\begin{displaymath}
g_1(x) = \prod_{i=0}^{n-1} \left( x-\alpha^{p^i} \right) \quad \quad g_2 (x) = \prod_{i=0}^{n-1} \left( x - (\alpha^{-1})^{p^i} \right).
\end{displaymath}

\begin{example}
Let $f(x) = x^3 + 3 x^2 + 2 \in \F_5[x]$. Then,
\begin{displaymath}
f^{R} (x) =  x^6 + x^5 + 3 x^4 + 3 x^3+ 3 x^2 + x+ 1.
\end{displaymath}
Following the steps explained above we want to decide if $f^{R}$ is irreducible or not and, in the latter case, factor it. Let $\beta$ be a root of $f$. We know that $f^{R}$ is irreducible if and only if $\beta^2-1$ is not a square in $\F_{5^3} \cong \F_5 [x] / (f)$. Since
\begin{displaymath}
(x^2-1)^{62} = 1 \quad \text{in $\F_5 [x] / (f)$},
\end{displaymath}
we conclude that $a=\beta^2 - 1$ is a quadratic residue in $\F_{5^3}$. Therefore, $f^{R}(x) = g_1 (x) \cdot g_2(x)$, where $g_1, g_2$ are two monic irreducible polynomials of degree $3$ of $\F_5 [x]$. Aiming to find the polynomials $g_1, g_2$, we look for an element 
\begin{displaymath}
c = c_2 \beta^2 + c_1 \beta + c_0 \in \F_{125}^*
\end{displaymath}
such that $c^5 = A \cdot c$, where $A = a^2$ and $a=\beta^2-1$. Expanding the last equation our problem is equivalent to finding three coefficients $c_0, c_1, c_2 \in \F_5$, not simultaneously equal to zero, such that
\begin{displaymath}
c_2 \cdot \beta^{10} + c_1 \beta^5 + c_0 = A \cdot (c_2 \beta^2 + c_1 \beta + c_0).
\end{displaymath} 
Expanding the left hand side of the last equation we get
\begin{eqnarray*}
c_2 \cdot (3 \beta^2 + 2 \beta +1 ) + c_1 \cdot (\beta^2+\beta+2) + c_0,
\end{eqnarray*}
while expanding the right hand side we get
\begin{eqnarray*}
 c_2 \cdot (\beta^2 + \beta + 1) + c_1 \cdot (2 \beta^2 +2 \beta +1) + c_0 \cdot (2 \beta^2 + 3 \beta +2).
\end{eqnarray*}
Therefore, solving $c^5 = A \cdot c$ amounts to solving the following linear system over $\F_5$:
\begin{displaymath}
\left\{
\begin{array}{rrr}
2 c_2 - c_1 - 2 c_0 & = & 0\\
c_2 - c_1 - 3 c_0 & = & 0\\
c_1 -c_0 & = & 0 
\end{array}
\right.
\end{displaymath}
It is easily seen that such a system has infinitely many solutions. More precisely, for a free choice of $c_0 \in \F_5$, the other coefficients are uniquely determined as $c_1 = c_0$ and $c_2 = - c_0$. For example, choosing $c_0 = 1$ we get $c_1 = 1$ and $ c_2 = -1$. Hence,
\begin{displaymath}
c = - \beta^2 + \beta +1 \quad \text{and} \quad \dfrac{c^2}{a} = 4.
\end{displaymath}
A square root of $4$ in $\F_5$ is $d = 2$. Therefore,
\begin{displaymath}
\dfrac{c}{d} = 2 \beta^2 - 2 \beta -2
\end{displaymath}
is a square root of $\beta^2-1$. We conclude that 
\begin{displaymath}
\alpha = 2 \beta^2 - \beta - 2
\end{displaymath}
is a solution of $\theta_{\frac{1}{2}}(x) = \beta$. Therefore, the polynomial
\begin{displaymath}
g_1 (x) = (x- \alpha) \cdot (x- \alpha^5) \cdot (x- \alpha^{25}) = x^3 + 3 x +3 
\end{displaymath}
is a monic irreducible factor of $f^{R}$. Now we can easily find the other factor and conclude that
\begin{displaymath}
f^{R} (x) = g_1 (x) \cdot g_2 (x) = (x^3+3x+3) \cdot (x^3+x^2+2). 
\end{displaymath}
\end{example}

\bibliography{Refs}

\providecommand{\bysame}{\leavevmode\hbox to3em{\hrulefill}\thinspace}
\providecommand{\MR}{\relax\ifhmode\unskip\space\fi MR }
\providecommand{\MRhref}[2]{%
  \href{http://www.ams.org/mathscinet-getitem?mr=#1}{#2}
}
\providecommand{\href}[2]{#2}
\begin{thebibliography}{10}

\bibitem{cho}
W.~Chou and I.~Shparlinski, \emph{On the the cycle structure of repeated
  exponentiation modulo a prime}, J. {N}umber {T}heory \textbf{107} (2004),
  no.~2, 345--356.

\bibitem{chu}
W.~M. Chu, \emph{Construction of irreducible polynomials using cubic
  transformation}, Appl. Algebra Engrg. Comm. Comput. \textbf{7} (1996), no.~1,
  15--19.

\bibitem{coh}
S.~D. Cohen, \emph{The explicit construction of irreducible polynomials over
  finite fields}, Des. Codes Cryptogr. \textbf{2} (1992), no.~2, 169--174.

\bibitem{gar}
T.~Garefalakis, \emph{On the action of ${G}{L}_2 (\mathbf{F}_q)$ on irreducible
  polynomials over $\mathbf{F}_q$}, J. Pure Appl. Algebra \textbf{215} (2011),
  no.~8, 1835--1843.

\bibitem{kyu1}
M.~K. Kyuregyan, \emph{Recurrent methods for constructing irreducible
  polynomials over {$\mathbf{F}_q$} of odd characteristics}, Finite Fields
  Appl. \textbf{9} (2003), no.~1, 39--58.

\bibitem{kyu2}
\bysame, \emph{Recurrent methods for constructing irreducible polynomials over
  {$\mathbf{F}_q$} of odd characteristics. {II}}, Finite Fields Appl.
  \textbf{12} (2006), no.~3, 357--378.

\bibitem{lan}
S.~Lang, \emph{Algebra}, Springer-Verlag, New York, 2002.

\bibitem{mey}
H.~Meyn, \emph{On the construction of irreducible self-reciprocal polynomials
  over finite fields}, Appl. Algebra Engrg. Comm. Comput. \textbf{1} (1990),
  no.~1, 43--53.

\bibitem{rog}
T.~G. Rogers, \emph{The graph of the square mapping on the prime fields},
  Discrete Math. \textbf{148} (1996), no.~1-3, 317--324.

\bibitem{sha}
M.~Sha and S.~Hu, \emph{Monomial dynamical systems of dimension over finite
  fields}, Acta Arith. \textbf{148} (2011), no.~4, 309--331.

\bibitem{sti}
H.~Stichtenoth and A.~Topuzo\u{g}lu, \emph{Factorization of a class of
  polynomials over finite fields}, Finite Fields Appl. \textbf{18} (2012),
  no.~1, 108--122.

\bibitem{seq2}
S.~Ugolini, \emph{Sequences of binary irreducible polynomials}, Preprint arXiv
  (2012).

\bibitem{SUk}
\bysame, \emph{On the iterations of certain map $x \mapsto k \cdot(x+x^{-1})$
  over finite fields of odd characteristic}, Preprint arXiv (2013).

\bibitem{vdw}
C.~van~de Woestijne, \emph{Deterministic equation solving over finite fields},
  Doctoral thesis, Leiden University, 2006,
  \url{http://hdl.handle.net/1887/4392}.

\bibitem{vas}
T.~Vasiga and J.~Shallit, \emph{On the iteration of certain quadratic maps over
  ${G}{F}(p)$}, Discrete Math. \textbf{277} (2004), no.~1-3, 219--240.

\end{thebibliography}
\end{document}